\documentclass[10pt,bezier]{article}
\usepackage{amsmath,amssymb,amsfonts,graphicx,caption,subcaption}
\usepackage[colorlinks,linkcolor=red,citecolor=blue]{hyperref}

\newtheorem{prethm}{{\bf Theorem}}[section]
\newenvironment{thm}{\begin{prethm}{\hspace{-0.5
em}{\bf.}}}{\end{prethm}}
\newtheorem{prepro}{{\bf Theorem}}

\newtheorem{precor}[prethm]{{\bf Corollary}}
\newenvironment{cor}{\begin{precor}{\hspace{-0.5
em}{\bf.}}}{\end{precor}}
\newtheorem{preconj}[prethm]{{\bf Conjecture}}

\newtheorem{preremark}[prethm]{{\bf Remark}}
\newenvironment{remark}{\begin{preremark}\em{\hspace{-0.5
em}{\bf.}}}{\end{preremark}}
\newtheorem{prelem}[prethm]{{\bf Lemma}}
\newenvironment{lem}{\begin{prelem}{\hspace{-0.5
em}{\bf.}}}{\end{prelem}}
\newtheorem{preque}[prethm]{{\bf Question}}

\newtheorem{preobserv}[prethm]{{\bf Observation}}

\newtheorem{predef}[prethm]{{\bf Definition}}

\newtheorem{preproposition}[prethm]{{\bf Proposition}}

\newtheorem{preproof}{{\bf Proof.}}
\newtheorem{preprooff}{{\bf Proof}}

\newenvironment{proof}[1]{\begin{preproof}{\rm
#1}\hfill{$\Box$}}{\end{preproof}}

\newtheorem{preproofs}{{\bf The second proof of }}

\newtheorem{preprooft}{{\bf Third proof of }}

\newtheorem{preproofF}{{\bf Proof of}}

\title{\bf\Large 
The existence of tree-connected $\{g,f\}$-factors in edge-connected graphs and tough graphs
}
\author{{\normalsize{\sc Morteza Hasanvand${}$} }\vspace{3mm}
\\{\footnotesize{${}$\it Department of Mathematical
 Sciences, Sharif
University of Technology, Tehran, Iran}}
{\footnotesize{}}\\{\footnotesize{ $\mathsf{morteza.hasanvand@alum.sharif.edu }$ }}}

\date{}

\def\p {p}
\def\q {q}

\def\t {t}

\begin{document}
\maketitle
\begin{abstract}{
In 1970 Lov{\'a}sz gave a necessary and sufficient condition for the existence of a factor $F$ in a graph $G$ such that for each vertex $v$, $g(v)\le d_F(v)\le f(v)$, where $g$ and $f$ are two integer-valued functions on $V(G)$ with $g\le f$. In this paper, we give a sufficient edge-connectivity condition for the existence of an $m$-tree-connected factor $H$ in a bipartite graph $G$ with bipartition $(X,Y)$ such that its complement is $m_0$-tree-connected and for each vertex $v$, $d_H(v)\in \{g(v),f(v)\}$, provided that for each vertex $v$, $g(v)+m_0\le \frac{1}{2}d_G(v)\le f(v)-m$ and $|f(v)-g(v)|\le k$, and there is $h(v)\in \{g(v),f(v)\}$ in which $\sum_{v\in X}h(v)=\sum_{v\in Y}h(v)$. Moreover, we generalize this result to general graphs. As an application, we give sufficient conditions for the existence of tree-connected $\{g,f\}$-factors in edge-connected graphs and tough graphs.
\\
\\
\noindent {\small {\it Keywords}: Factor; vertex degree; $\{g,f\}$-factor; connectivity; toughness. }} {\small
}
\end{abstract}
%
%
%
%
%
%
%
%
%
%
\section{Introduction}
In this article, all graphs may have loops and multiple edges and a simple graph have neither loops not multiple edges.
Let $G$ be a graph. The vertex set, the edge set of $G$, and the number of components of $G$ are denoted by $V(G)$, $E(G)$, and $\omega(G)$, respectively.
We denote by $d_G(v)$ the degree of a vertex $v$ in the graph $G$.
For a vertex set $X$, we denote by $d_G(X)$ the number of edges of $G$ with exactly one end in $X$,
 and denote by
$e_G(X)$ denotes the number of edges of $G$ with both ends in $X$.
We also denote by $d_G(X,Y)$ the number of edges of $G$ with one end in $X$ and the other one in $Y$ and denote by $G[X, Y]$ the induced bipartite factor of $G$ with the bipartition $(X,Y)$.
Let $f:V(G)\rightarrow Z_k$ be a mapping, where $Z_k$ is the cyclic group of order $k$.
We say that $f$ is 
{\it  compatible with $G$}, if  for every bipartition $X,Y$ of $V(G)$, there are two integers $x$ and $y$ satisfying
 $0\le x\le e_G(X)$ and $0\le y\le e_G(Y)$ such that 
$\sum_{v\in X} f(v) -2x\stackrel{k}{\equiv} \sum_{v\in Y}f(v)-2y$.
The {\it bipartite index $bi(G)$} of a graph $G$ is the smallest number of all $|E(G)\setminus E(H)|$ taken over all bipartite spanning induced subgraphs~$H$ of $G$. 
For each vertex, let $L(v)$ be a set of integers.
An orientation of $G$ is called {\it $L$-orientation}, if for each vertex $v$, $d^+_G(v)\in L(v)$, where $d^+_G(v)$ denotes the out-degree of $v$.
A factor $F$ of $G$ is said to be 
(i) {\it $L$-factor}, if for each vertex $v$, $d_F(v)\in L(v)$,
 (ii) {\it $f$-factor}, if for each vertex $v$, $d_F(v)= f(v)$,
 (iii) {\it $(g,f)$-factor}, if for each vertex $v$, $g(v)\le d_F(v)\le f(v)$,
where $g$ and $f$ are two integer-valued functions on $V(G)$.
We denote by $\omega_{g,f}(G, A,B)$ the number of components $C$ of $G\setminus (A\cup B)$
such that for each vertex $v\in V(C)$, $g(v)=f(v)$ and $d_{G}(V(C),B)\not\stackrel{2}{\equiv}\sum_{v\in V(C)}f(v)$.
For convenience, we write $\omega_{f}(G, A,B)$ instead of $\omega_{g,f}(G, A,B)$ when $g=f$.
A graph $G$ is called 
{\it $m$-tree-connected}, if it contains $m$ edge-disjoint spanning trees. 
Note that by the result of Nash-Williams~\cite{Nash-Williams-1961} and Tutte~\cite{Tutte-1961} every $2m$-edge-connected graph is $m$-tree-connected.
Throughout this article, all variables $k$ and $m$ are nonnegative integers.

%

In 1952 Tutte constructed the following criterion for the existence of $f$-factors. 

\begin{thm}{\rm (\cite{Tutte-1952})}\label{intro:thm:Tutte}
{Let $G$ be a graph and let $f$ be an integer-valued function on $V(G)$. 
Then $G$ has an $f$-factor, if and only if 
for all disjoint subsets $A$ and $B$ of $V(G)$,
$$\omega_{f}(G, A,B)\le \sum_{v\in A} f(v)+\sum_{v\in B} (d_{G\setminus A}(v)-f(v)).$$
}\end{thm}

In 1970 Lov{\'a}sz generalized Tutte's result to the following bounded degree version.

\begin{thm}{\rm (\cite{Lovasz-1970})}\label{lem:Lovasz}
{Let $G$ be a graph and let $g$ and $f$ be two integer-valued functions on $V(G)$ with $g\le f$. 
Then $G$ has a $(g,f)$-factor, if and only if 
for any disjoint subsets $A$ and $B$ of $V(G)$,
$$\omega_{g,f}(G, A,B)\le \sum_{v\in A} f(v)+\sum_{v\in B} (d_{G\setminus A}(v)-g(v)).$$
}\end{thm}

In 2014 Thomassen~\cite{Thomassen-2014} introduced modulo factors and gave a sufficient edge-connectivity for the existence of $f$-factors modulo $k$ in bipartite graphs. Next, his result was developed to graphs with bipartite index at least $k-1$ by 
Thomassen, Wu, and Zhang (2016)~\cite{Thomassen-Wu-Zhang-2016}.
Recently, the present author developed their results to a bounded degree version and derived the following theorem.
In Section~\ref{sec:factors}, we generalize this result in order to guarantee the existence of $\{g,f\}$-factors in edge-connected graphs as mentioned in the abstract.
\begin{thm}{\rm (\cite{ModuloFactorBounded})}\label{intro:thm:modulo-factor}
{Let $G$ be a $6k$-tree-connected graph, $k\ge 1$, and let $f$ be a positive integer-valued function on $V(G)$. 
Assume that $f$ is compatible with $G$ (modulo $k$).
If for each vertex $v$, $f(v)\le \frac{1}{2}d_G(v)\le f(v)+k$,
then $G$ has a factor $F$ such that for each vertex $v$,
$$d_F(v)\in \{f(v),f(v)+k\}.$$
}\end{thm}
In~\cite{complementary}, we developed Theorem~\ref{intro:thm:modulo-factor} to the following tree-connected version. 
In Section~\ref{sec:tree-connected-factors}, we generalize this result in order to guarantee the existence of tree-connected $\{g,f\}$-factors in edge-connected graphs.
\begin{thm}{\rm (\cite{complementary})}\label{intro:thm:f-f+k}
{Let $G$ be a $(2m+2m_0+6k)$-tree-connected graph, $k\ge m+m_0> 0$, and let $f$ be a positive integer-valued function on $V(G)$. Assume that $f$ is compatible with $G$ (modulo $k$).
If for each vertex $v$, $f(v)+m_0\le \frac{1}{2}d_G(v)\le f(v)+k-m$,
then $G$ has an $m$-tree-connected factor $H$ such that its complement is $m_0$-tree-connected and for each vertex $v$,
$$d_H(v)\in \{f(v),f(v)+k\}.$$
}\end{thm}

In 1990 Katerinis~\cite{Katerinis-1990} gave a sufficient toughness condition for the existence of $f$-factors.
By applying his result, we formulated the following sufficient toughness condition for the existence of tree-connected $\{f, f+1\}$-factors. 
\begin{thm}{\rm (\cite{ClosedTrails})}\label{intro:thm:lem:f-f+1}
{Let $G$ be a graph, let $b$ be a positive integer, and 
let $f$ be a positive integer-valued function on $V(G)$ satisfying $2m\le f\le b$.
If $G$ is $b^2$-tough and $|V(G)|\ge b^2$, then it has an $m$-tree-connected $\{f, f+1\}$-factor. 
}\end{thm}

In~\cite{f-f+k}, we formulated the following sufficient toughness condition for the existence of tree-connected $\{f, f+k\}$-factors  using Theorem~\ref{intro:thm:f-f+k}. 
In Section~\ref{sec:tough-graphs}, by applying the recent development of Theorem~\ref{intro:thm:f-f+k}, we derive a sufficient toughness condition for the existence of tree-connected $\{g,f\}$-factors.
\begin{thm}{\rm (\cite{f-f+k})}\label{intro:thm:tree-connected}
{Let $G$ be a graph and let $f$ be a positive integer-valued function on $V(G)$ 
satisfying $3m+2m_0+6k<  f +k\le b$ and $m+m_0<k$, where $k$, $b$, $m$, and $m_0$ are four nonnegative integers.
Assume that $(k-1)\sum_{v\in V(G)}f(v)$ is even. If $G$ is $4b^2$-tough and $|V(G)|\ge 4b^2$, then $G$ has an $m$-tree-connected factor $H$ such that its complement is $m_0$-tree-connected and for each vertex $v$, $$d_H(v)\in \{f(v),f(v)+k\}.$$
}\end{thm}
%
%
%
%
%
%
%
%
\section{Preliminary results}

\subsection{Factors in Eulerian graphs with bounded bipartite index}
In this subsection, we shall state a result that is useful for finding factors in Eulerian graphs whose degrees are close to the half of the correspondence degrees in the main graph.
For this purpose, we need the following reformulation of Theorem~\ref{intro:thm:Tutte}.
\begin{lem}{\rm (\cite{Tutte-1952})}\label{lem:Tutte}
{Let $G$ be a connected graph and let $f$ be an integer-valued function on $V(G)$ with $\sum_{v\in V(G)}f(v)$ even. 
Then $G$ has an $f$-factor, if and only if 
for all disjoint subsets $A$ and $B$ of $V(G)$ with $A\cup B\neq \emptyset$,
$$\omega_{f}(G, A,B)<2+ \sum_{v\in A} f(v)+\sum_{v\in B} (d_{G\setminus A}(v)-f(v)).$$
}\end{lem}
Now, we are ready to prove the main result of this subsection.
\begin{thm}\label{thm:Eulerian}
{Let $G$ be an Eulerian graph and let $i$ be an integer-valued function on $V(G)$ with
 $|E(G)| \stackrel{2}{\equiv} t$, where $t=\sum_{v\in V(G)}|i(v)|$.
If $G$ is $(2t-1)$-edge-connected and $bi(G)\ge t-1$, then $G$ has a factor $F$ 
such that for each vertex $v$, $$d_F(v)=d_G(v)/2+i(v).$$
}\end{thm}
\begin{proof}
{For each vertex $v$, define $f(v)=d_G(v)/2+i(v)$.
By the assumption, $\sum_{v\in V(G)}f(v)$ must be even.
Let $A$ and $B$ be two disjoint subsets of $V(G)$ with $A\cup B\neq \emptyset$.
 If  $A\cup B\neq V(G)$, then since $G$ is $(2t-1)$-edge-connected, it is not hard to check that
$$\omega(G\setminus (A\cup B))\le \sum_{v\in A\cup B}d_G(v)/2-d_G(A,B)-(2t-1)/2+1.$$
 If $A\cup B=V(G)$, then since $bi(G)\ge t-1$, we must have $e_G(A)+e_G(B)\ge t-1$ which implies that
$$\omega(G\setminus (A\cup B))
=0= 
 \sum_{v\in V(G)}d_G(v)/2-|E(G)|\le \sum_{v\in A\cup B}d_G(v)/2-d_G(A,B)-(t-1).$$
Since $t\ge \sum_{v\in A\cup B}|i(v)|\ge \sum_{v\in A}i(v)+\sum_{v\in B}-i(v)$, in both cases, we therefore have
$$\omega(G\setminus (A\cup B)) \le \sum_{v\in A}(d_G(v)/2+i(v))+\sum_{v\in B}(d_G(v)/2-i(v))-d_G(A,B)+3/2,$$
which implies that
$$\omega(G\setminus (A\cup B))
< 2+ \sum_{v\in A}f(v)+\sum_{v\in B}(d_G(v)-f(v))-d_G(A,B).$$
Hence the assertion follows from Lemma~\ref{lem:Tutte}.
}\end{proof}
The following special case plays an important role in this paper.
This result was formerly  proved in \cite[Theorem 4.11]{ModuloFactorBounded} for $6|t|$-tree-connected graphs.
\begin{cor}\label{cor:Eulerian}
{Let $G$ be an Eulerian with $z\in V(G)$ and let $t$ be an integer number with $t\stackrel{2}{\equiv}\sum_{v\in V(G)}|E(G)|$
If $G$ is $2|t|$-tree-connected and $bi(G)\ge |t|-1$, then $G$ has a factor $F$ 
such that $d_F(z)=d_G(z)/2+t$ and $d_F(v)=d_G(v)/2$ for each $v\in V(G)\setminus \{z\}$.
}\end{cor}
\subsection{Eulerian tree-connected factors with bounded bipartite index}
Before applying Corollary~\ref{cor:Eulerian}, we need to utilize the following simple tool in order to establish a desired spanning Eulerian subgraph.
\begin{lem}\label{lem:decomposition:Eulerian}
{Let $G$ be a graph. If $G[X,Y]$ is $(m_1+m_2+1)$-tree-connected graph for a bipartition $X,Y$ of $V(G)$, 
then $G$ can be decomposed into two factors $G_1$ and $G_2$ 
such that $G_1$ is a $m_1$-tree-connected bipartite graph with bipartition $(X,Y)$, and $G_2$ is an $m_2$-tree-connected Eulerian graph.
}\end{lem}
\begin{proof}
{Decompose $G[X,Y]$ into three factors $T$, $H_1$, and $H_2$ such that $T$ is a tree, $H_1$ is $m_1$-tree-connected, and $H_2$ is $m_2$-tree-connected. Let $F$ be a spanning forest of $T$ such that for each vertex $v$, $d_F(v)$ and $d_{H_2}(v)$ have the same parity. Now, it is enough to define $G_2=H_2\cup F$ and $G_1=G\setminus E(G_2).$
}\end{proof}
We will apply the following lemma in Section~\ref{sec:tree-connected-factors} instead of the above-mentioned lemma.
\begin{lem}{\rm (\cite{complementary})}\label{lem:decomposition:bipartite-index}
{If $G$ is a $(2m_1+2m_2)$-tree-connected graph and $m_2\ge k_0\ge 0$, then $G$ can be decomposed into two factors $G_1$ and $G_2$ such that $G_1$ is a $2m_1$-edge-connected Eulerian graph, $G_2[X,Y]$ is $m_2$-tree-connected for a bipartition $X,Y$ of $V(G)$, 
$$e_{G_2}(X)+e_{G_2}(Y)\ge \min\{k_0, bi(G)\}.$$ 
}\end{lem}
The following simple lemma can also help use to give a bound on bipartite index.
\begin{lem}{\rm (see \cite{ModuloFactorBounded})}\label{lem:bound:bipartite-index}
{Let $G$ be a graph. If for a bipartition $X,Y$ of $V(G)$, 
the graph $G[X,Y]$ is $k$-tree-connected and $e_G(X)+e_G(Y)\ge k$,
 then $bi(G)\ge k$. 
}\end{lem}
\begin{proof}
{Let $T_1,\ldots, T_k$ be $k$ edge-disjoint spanning trees of $G[X,Y]$ and let $e_1,\ldots, e_k$ be $k$ distinct edges of the graph $G[X]\cup G[Y]$. Since $T_i$ is a bipartite graph with the bipartition $(X,Y)$, the graph $T_i+e_i$ must contain an odd cycle $C_i$.
Therefore, $G$ contains $k$ edge-disjoint odd cycles and so $bi(G)\ge k$.
}\end{proof}
\section{The existence of $\{g,f\}$-factors}
\label{sec:factors}
In this section, we shall provide some sufficient condition conditions for the existence of $\{g,f\}$-factors in highly edge-connected graphs.
\subsection{Tools: orientations}
We shall below recall some recent results about the existence of orientations with constrained out-degrees in highly edge-connected graphs. For applying them in our proofs, we need the following lemma to make factors with constrained degrees from such orientations.
\begin{lem}{\rm (see \cite{p-p+1-q-1-q})}\label{lem:factor-orientation}
{Let $G$ be a bipartite graph with bipartition $(X,Y)$ and $L:V(G)\rightarrow 2^\mathbb{Z}$ be a function.
Then $G$ admits an $L$-orientation if and only if $G$ admits an $L_0$-factor, where for each vertex $v$,
$$L_0(v)=
 \begin{cases}
L(v),	&\text{when $v\in X$};\\
\{d_G(v)-i:i\in L(v)\},	&\text{when $v\in Y$}. 
\end {cases}$$
}\end{lem}
\begin{proof}
{If $D$ is an orientation of $G$, then the factor $F$ consisting of all edges of $G$ directed from $X$ to $Y$
satisfies $d_F(v)=d^+_D(v)$ for each $v\in X$, and $d_F(v)=d_G(v)-d^+_D(v)$ for each $v\in Y$.
Conversely, from every factor $F$, we can make an orientation $D$ whose edges directed from $X$ to $Y$ 
are exactly the same edges of $F$.
}\end{proof}
\begin{thm}{\rm (\cite{p-q})}\label{thm:exact-p,q}
{Let $G$ be a $4k^2$-tree-connected graph and let $\p$ and $\q$ be two integer-valued functions on $V(G)$ in which for each vertex $v$, $\p (v)\le d_G(v)/2\le\q(v)$ and $|\q(v)-\p(v)|\le k$.
Then $G$ has an orientation such that for each $v\in V(G)$,
$$d^+_G(v)\in \{\p(v),\q(v)\},$$
if and only if there is an integer-valued function $t$ on $V(G)$ in which $\t(v)\in \{\p(v),\q(v)\}$ for each $v\in V(G)$, and 
$|E(G)|=\sum_{v\in V(G)}\t(v)$.
Furthermore, for an arbitrary given vertex $z$, we can have $d^+_G(z)=\t(z)$.
}\end{thm}
\begin{thm}{\rm (\cite{p-q})}\label{thm:z-defective-p,q}
{Let $G$ be a graph with $z\in V(G)$, let $k$ be a positive integer, and let $\p$ and $\q$ be two integer-valued functions on $V(G)$ in which for each vertex $v$, $\p(v)\le d_G(v)/2\le \q(v)$ and $|\q(v)-\p(v)|\le k$.
If $G$ is $(\frac{3}{2}k+1)(k-1)$-tree-connected, then it has an orientation such that for each $v\in V(G)\setminus \{z\}$,
$$d^+_G(v)\in \{\p(v),\q(v)\}.$$
Furthermore, for the vertex $z$, we can have 
$-x\le d^+_G(z)-d_G(z)/2<k-x$, where $x$ is an arbitrary real number $x\in [0, k)$.
}\end{thm}
\subsection{Bipartite graphs}
The following theorem gives a necessary and sufficient condition for the existence of $\{g,f\}$-factors in highly edge-connected bipartite graphs provided that $f$ and $g$ are close enough.
\begin{thm}\label{thm:factor:necessary-sufficient:bipartite}
{Let $G$ be a $4k^2$-tree-connected bipartite graph with bipartition $(X,Y)$ and let $g$ and $f$ be two integer-valued functions on $V(G)$ in which for each $v\in V(G)$, $g(v)\le d_G(v)/2\le f(v)$ and $|f(v)-g(v)|\le k$.
Then $G$ has a factor $F$ such that for each vertex $v$,
$$d_F(v)\in \{g(v),f(v)\},$$
if and only if 
 there is an integer-valued function $h$ on $V(G)$ in which 
$h(v)\in \{g(v),f(v)\}$ for each $v\in V(G)$, and 
$\sum_{v\in X}h(v)=\sum_{v\in Y}h(v)$. 
Furthermore, for an arbitrary given vertex $z$, we can have $d_F(z)=h(z)$.
}\end{thm}
\begin{proof}
{If $G$ contains a $\{g,f\}$-factor $F$, then we can find the desired function $h$ by defining $h(v)=d_F(v)$ for each vertex $v$. 
Note that  $\sum_{v\in X}h(v)=\sum_{v\in X}d_F(v)=|E(F)|=\sum_{v\in Y}d_F(v)=\sum_{v\in Y}h(v)$.
Now, assume there is an integer-valued function $h$ on $V(G)$ in which 
$h(v)\in \{g(v),f(v)\}$ for each vertex $v$, and 
$\sum_{v\in X}h(v)=\sum_{v\in Y}h(v)$. 
We may assume that $z\in X$.
For each $v\in X$, define $\p(v)=g(v)$ and $\q(v)=f(v)$, 
and for each $v\in Y$, define $\p(v)=d_G(v)-f(v)$ and $\q(v)=d_G(v)-g(v)$.
If  we define $t(v)=h(v)$ for each $v\in X$, and define $t(v)=d_G(v)-h(v)$ for each $v\in Y$, then 
 $t(v)\in \{p(v), q(v)\}$ for each vertex $v$, and 
$\sum_{v\in V(G)}t(v)=\sum_{v\in X}h(v) +\sum_{v\in Y}(d_G(v)-h(v))= |E(G)|$. 
Thus by Theorem~\ref{thm:exact-p,q}, the graph $G$ has an orientation $D$ 
such that for each vertex $v$, $d^+_D(v)\in \{p(v), q(v)\}$, and $d^+_D(z)=t(z)$.
Consequently, by Lemma~\ref{lem:factor-orientation}, the graph $G$ must contain a $\{g,f\}$-factor $F$ with $d_F(z)=h(z)$.
Hence the proof is completed.
}\end{proof}
\subsection{Almost bipartite graphs}
In the following theorem, we generalize 
Theorem~\ref{thm:factor:necessary-sufficient:bipartite} to general graphs by imposing
more flexible conditions on the function $h$. 
More precisely, we will use an advantage of odd cycles and loops in our proof which allows us to modify the degree of some vertices a little.
\begin{thm}\label{thm:almost:bipartite}
{Let $G$ be a graph and let $g$ and $f$ be two integer-valued functions on $V(G)$ in which for each $v\in V(G)$, $g(v)\le d_G(v)/2\le f(v)$ and $|f(v)-g(v)|\le k$.
Assume that $G[X,Y]$ is $(4k^2+2k)$-tree-connected and $e_G(X)+e_G(Y)\le k-1$ for a bipartition $X,Y$ of $V(G)$.
Then $G$ has a factor $F$ such that for each vertex $v$,
$$d_F(v)\in \{g(v),f(v)\}.$$
if there is an integer-valued function $h$ on $V(G)$ in which 
for each $v\in V(G)$, $h(v)\in \{g(v),f(v)\}$, $\sum_{v\in V(G)}h(v)$ is even, and 
$2e_G(X)+1\ge \sum_{v\in X}h(v)-\sum_{v\in Y}h(v)\ge 0$.
}\end{thm}
\begin{proof}
{By Lemma~\ref{lem:decomposition:Eulerian}, the graph $G$ can be decomposed into two factors $G_1$ and $G_2$ such that $G_1$ is a $4k^2$-tree-connected bipartite graph with bipartition $(X,Y)$ and 
$G_2[X,Y]$ is a $2k$-tree-connected Eulerian graph so that $e_{G_2}(X)+e_{G_2}(Y)= e_{G}(X)+e_{G}(Y)\le k-1$.
Pick $z\in X$ and let $t=\sum_{v\in X}h(v)-\sum_{v\in Y}h(v)-e_G(X)+e_G(Y)$.
Since $\sum_{v\in X}h(v)-\sum_{v\in Y}h(v)$ is even, we must have
$2e_G(X)\ge \sum_{v\in X}h(v)-\sum_{v\in Y}h(v)$ and 
$2e_G(Y)\ge 0\ge \sum_{v\in Y}h(v)-\sum_{v\in X}h(v)$ which imply that $|t| \le e_{G}(X)+e_{G}(Y)\le k-1$.
Define $h'(z)=h(z)-d_{G_2}(z)/2-t$, and define $h'(v)=h(v)-d_{G_2}(v)/2$ for each $v\in V(G)\setminus \{z\}$.
It is easy to check that
$\sum_{v\in X}d_{G_2}(v)/2+e_{G_2}(X)=\sum_{v\in Y}d_{G_2}(v)/2+e_{G_2}(Y)$,
which implies that
$$\sum_{v\in X}h'(v)=\sum_{v\in X}(h(v)-d_{G_2}(v)/2)-t=
\sum_{v\in Y}(h(v)-d_{G_2}(v)/2)=\sum_{v\in Y}h'(v).$$
Thus by Theorem~\ref{thm:factor:necessary-sufficient:bipartite},
the graph $G_1$ has a factor $F_1$
such that 
$d_{F_1}(z)\in \{g(z)-d_{G_2}(z) /2-t,f(z)-d_{G_2}(z) /2-t\}\supseteq\{h'(z)\} $ and 
$d_{F_1}(v)\in \{g(v)-d_{G_2}(v) /2,f(v)-d_{G_2}(v) /2\}\supseteq\{h'(v)\} $ for each $v\in V(G)\setminus \{z\}$.
Since $t$ and $e_{G}(X)+e_{G}(Y)$ have the same parity, we must have 
$t\stackrel{2}{\equiv} e_{G_2}(X)+e_{G_2}(Y) \stackrel{2}{\equiv}\sum_{v\in V(G)}d_{G_2}(v)/2$.
Thus by applying Corollary~\ref{cor:Eulerian}, the graph $G_2$ has a factor $F_2$ such that 
$d_{F_2}(z)=d_{G_2}(z)/2+t$ and 
$d_{F_2}(v)=d_{G_2}(v)/2$ for each $v\in V(G)\setminus \{z\}$.
Note that according to Lemma~\ref{lem:bound:bipartite-index}, $|t|\le bi(G_2)\le k-1$.
It is easy to check that $F_1\cup F_2$ is the desired factor we are looking for.
}\end{proof}
\subsection{Graphs with bipartite index at least $k-1$}
The following theorem provides a supplement for Theorem~\ref{thm:almost:bipartite}.
This result replaces a simpler condition for graphs with higher bipartite index.
\begin{thm}\label{thm:bi-index:large}
{Let $G$ be a graph, let $k$ be a positive integer, and let $g$ and $f$ be two integer-valued functions on $V(G)$ in which for each $v\in V(G)$, $g(v)\le d_G(v)/2\le f(v)$ and $|f(v)-g(v)|\le k$.
Assume that $G[X,Y]$ is $3k^2$-tree-connected and $e_G(X)+e_G(Y)\ge k-1$ for a bipartition $X,Y$ of $V(G)$.
Then $G$ has a factor $F$ such that for each vertex~$v$,
$$d_F(v)\in \{g(v),f(v)\},$$
if and only if either there is a vertex $u$ with $f(u)-g(u)$ odd or $f(v)-g(v)$ is even for all vertices $v$ and $\sum_{v\in V(G)}f(v)$ is even.
}\end{thm}
\begin{proof}
{By Lemma~\ref{lem:decomposition:Eulerian}, the graph $G$ can be decomposed into two factors $G_1$ and $G_2$ such that $G_1$ is a $(\frac{3}{2}k+1)(k-1)$-tree-connected bipartite graph with bipartition $(X,Y)$ and 
$G_2[X,Y]$ is a $2k$-tree-connected Eulerian graph and $e_{G_2}(X)+e_{G_2}(Y)=e_{G}(X)+e_{G}(Y)\ge k-1$.
Note that $(\frac{3}{2}k+1)(k-1)+2k+1\le 3k^2$.
If for all vertices $v$, $f(v)-g(v)$ is even, 
we set $z$ to be an arbitrary vertex; 
otherwise, we set $z$ to be a vertex with $f(z)-g(z)$ odd.
By applying a combination of Lemma~\ref{lem:factor-orientation} and Theorem~\ref{thm:z-defective-p,q}, similarly to the proof of Theorem~\ref{thm:factor:necessary-sufficient:bipartite}, one can conclude that the graph $G_1$ has a factor $F_1$
such that for each $v\in V(G)\setminus \{z\}$, $d_{F_1}(v)\in \{g(v)-d_{G_2}(v) /2,f(v)-d_{G_2}(v) /2\}$, and 
$$d_G(v)/2-k\le (f(z)+g(z))/2-d_{G_2}(z)/2-k/2 \le \, d_{F_1}(z)\,\le (f(z)+g(z))/2-d_{G_2}(z)/2+k/2\le d_G(v)/2+k.$$
Let $t\in \{g(z)-d_{F_1}(z)-d_{G_2}(z)/2, f(z)-d_{F_1}(z)-d_{G_2}(z)/2\}$.
Note that $|t|\le |f(z)-g(z)|/2+k/2\le k$.
If for all vertices $v$, $f(v)-g(v)$ is even, then 
$$ d_{F_1}(z) +\sum_{v\in V(G)\setminus \{z\}}d_{G_2}(v)/2
\stackrel{2}{\equiv}\sum_{v\in V(G)\setminus \{z\}}(d_{F_1}(v)+d_{G_2}(v)/2)
 \stackrel{2}{\equiv}\sum_{v\in V(G)\setminus \{z\}}f(v)\stackrel{2}{\equiv} f(z), $$
which implies that
$t \stackrel{2}{\equiv} f(z)-d_{F_1}(z)-d_{G_2}(z)/2\stackrel{2}{\equiv}\sum_{v\in V(G)}d_{G_2}(v)/2$.
Therefore, we can select $t$ such that $t\stackrel{2}{\equiv}\sum_{v\in V(G)}d_{G_2}(v)/2$, 
regardless of $f(z)-g(z)$ is odd or not.
Thus by applying Corollary~\ref{cor:Eulerian}, the graph $G_2$ has a factor $F_2$ such that 
$d_{F_2}(z)=d_{G_2}(z)/2+t$ and 
$d_{F_2}(v)=d_{G_2}(v)/2$ for each $v\in V(G)\setminus \{z\}$.
Note that according to Lemma~\ref{lem:bound:bipartite-index}, $bi(G_2)\ge k-1 \ge |t|-1$.
It is easy to check that $F_1\cup F_2$ is the desired factor we are looking for.
}\end{proof}
\section{The existence of tree-connected $\{g,f\}$-factors}
\label{sec:tree-connected-factors}
In this section, we are going to develop
Theorems~\ref{thm:factor:necessary-sufficient:bipartite} and~\ref{thm:bi-index:large} to tree-connected versions. For this purpose, we need to apply the following lemma in our proofs.
\begin{lem}{\rm (\cite{complementary})}\label{lem:m-m0}
{Every $(2m+2m_0)$-edge-connected graph $G$ with $m+m_0>0$ 
has an $m$-tree-connected $H$ such that its complement is $m_0$-tree-connected and for each vertex~$v$,
$$\lfloor \frac{d_G(v)}{2}\rfloor -m_0\le d_H(v) \le \lceil \frac{d_G(v)}{2}\rceil+m.$$
}\end{lem}
\subsection{Bipartite graphs}
A tree-connected version of Theorem~\ref{thm:factor:necessary-sufficient:bipartite} is given in the following theorem.
\begin{thm}
{Let $G$ be a $(2m+2m_0+4k^2)$-tree-connected bipartite graph with bipartition $(X,Y)$ and
let $g$ and $f$ be two integer-valued functions on $V(G)$ in which for each $v\in V(G)$, 
$g(v)+m_0\le \frac{1}{2}d_G(v)\le f(v)-m$ and $|f(v)-g(v)|\le k$.
Then 
 $G$ has an $m$-tree-connected factor $H$ such that its complement is $m_0$-tree-connected and for each vertex~$v$, $$d_H(v)\in \{g(v),f(v)\},$$
if and only if there is an integer-valued function $h$ on $V(G)$ with $\sum_{v\in X}h(v)=\sum_{v\in Y}h(v)$
 in which for each vertex $v$, $h(v)\in \{g(v),f(v)\}$. Furthermore, for an arbitrary given vertex $z$, we can have $d_H(z)=h(z)$.
}\end{thm}
\begin{proof}
{If $m=m_0=0$, then the assertion follows from Theorem~\ref{thm:factor:necessary-sufficient:bipartite}.
So, suppose $m+m_0>0$. 
Since every $2$-tree-connected graph has a spanning Eulerian subgraph~\cite{Jaeger-1979}, one can decompose 
 $G$ into two factors $G_1$ and $G_2$ such that $G_1$ is a $(2m+2m_0)$-edge-connected Eulerian graph and 
$G_2$ is $4k^2$-tree-connected. 
By Lemma~\ref{lem:m-m0}, 
the graph $G_1$ has an $m$-tree-connected factor $H'$ such that its complement is $m_0$-tree-connected and for each vertex $v$, 
$d_{G_1}(v)/2-m_0 \le d_{H'}(v)\le d_{G_1}(v)/2+m$.
For each vertex $v$, define $g'(v)=g(v)-d_{H'}(v)$ and $f'(v)=f(v)-d_{H'}(v)$ so that $|f'(v)-g'(v)|\le k$.
By the assumption, 
$$g'(v)=g(v)-d_{H'}(v)\le d_{G}(v)/2-m_0-(d_{G_1}(v)/2-m_0)= d_{G_2}(v)/2.$$
Similarly, $d_{G_2}(v)/2\le f'(v)$.
If for each vertex $v$, we define $h'(v)=h(v)-d_{H'}(v)$, then 
$h(v)\in \{g'(v),f'(v)\}$, and also
$\sum_{v\in X}h'(v)=\sum_{v\in Y}h'(v)$.
Thus by Theorem~\ref{thm:factor:necessary-sufficient:bipartite}, the graph $G_2$ has a factor $F$ such that for each vertex $v$,
$d_{F(v)}\in \{g'(v),f'(v)\}$.
It is easy to check that $ H'\cup F$ is the desired factor we are looking for.
}\end{proof}
\subsection{Graphs with bipartite index at least $k-1$}
A tree-connected version of Theorem~\ref{thm:bi-index:large} is given in the following theorem.
\begin{thm}\label{thm:main:g-f}
{Let $G$ be a $(2m+2m_0+6k^2)$-tree-connected graph with $bi(G)\ge k-1$ and
let $g$ and $f$ be two integer-valued functions on $V(G)$ in which for each $v\in V(G)$, 
$g(v)+m_0\le \frac{1}{2}d_G(v) \le f(v)-m$ and $|f(v)-g(v)|\le k$.
Then $G$ has an $m$-tree-connected factor $H$ such that its complement is $m_0$-tree-connected and for each vertex~$v$,
 $$d_H(v)\in \{g(v),f(v)\}.$$
if and only if either there is a vertex $u$ with $f(u)-g(u)$ odd or $f(v)-g(v)$ is even for all vertices $v$ and $\sum_{v\in V(G)}f(v)$ is even. 
}\end{thm}
\begin{proof}
{By Lemma~\ref{lem:decomposition:bipartite-index}, the graph $G$
 can be decomposed into two factors $G_1$ and $G_2$ such that 
$G_1$ is a $(2m+2m_0)$-edge-connected Eulerian graph so that 
$G_2[X,Y]$ is $2k^2$-tree-connected and $e_{G_2}(X)+e_{G_2}(Y)\ge k-1$ for a bipartition $X,Y$ of $V(G)$. 
If $m=m_0=0$, then the assertion follows from Theorem~\ref{thm:bi-index:large}.
So, suppose $m+m_0>0$. 
By Lemma~\ref{lem:m-m0}, 
the graph $G_1$ has an $m$-tree-connected factor $H'$ such that its complement is $m_0$-tree-connected and for each vertex $v$, 
$d_{G_1}(v)/2-m_0 \le d_{H'}(v)\le d_{G_1}(v)/2+m$.
For each vertex $v$, define $g'(v)=g(v)-d_{H'}(v)$ and $f'(v)=f(v)-d_{H'}(v)$ so that $|f'(v)-g'(v)|\le k$. 
By the assumption, 
$$g'(v)=g(v)-d_{H'}(v)\le d_{G}(v)/2-m_0-(d_{G_1}(v)/2-m_0)= d_{G_2}(v)/2.$$
Similarly, $d_{G_2}(v)/2\le f'(v)$.
In addition, either there is a vertex $u$ with $f'(u)-g'(u)$ odd or $f'(v)-g'(v)$ is even for all vertices $v$ and $\sum_{v\in V(G)}f'(v)$ is even.
Thus by Theorem~\ref{thm:bi-index:large}, the graph $G_2$ has a factor $F$ such that for each vertex $v$,
$d_{F(v)}\in \{g'(v),f'(v)\}$.
It is easy to check that $ H'\cup F$ is the desired factor we are looking for.
}\end{proof}
\begin{remark}
{Note that the results of this section can be developed to version for investigating tree-connected factors with given sparse lists on degrees; similar to a result in~\cite{p-q} for investigating orientations with given sparse lists on out-degrees. 
Also, Theorem~\ref{thm:main:g-f} can be stated for graphs with $bi(G)\le k-1$, but we need to insert a stronger condition on $h$ using the same proof.
}\end{remark}
\section{An application to tough enough graphs}
\label{sec:tough-graphs}
The following lemma played a simple but important role in the proof of Theorem~\ref{intro:thm:tree-connected}. 
In order to make a new generalization for it, we again need to apply this lemma in our proof.
\begin{lem}{\rm (\cite{f-f+k})}\label{lem:factor-bipartite-index}
{Let $G$ be a graph with a $(2k-2)$-tree-connected factor $F$. 
If $G$ is $4k$-tough and $|V(G)| \ge 4k$, then it has a matching $M$ of size $k-1$ such that $bi(F\cup M)\ge k-1$.
}\end{lem}
A generalization of Theorem~\ref{intro:thm:tree-connected} is given in the following theorem.
\begin{thm}
{Let $G$ be a graph and 
let $g$ and $f$ be two integer-valued functions on $V(G)$ satisfying 
$3m+2m_0+6k^2<f\le b$ and $m+m_0<f-g\le k$, 
where $k$, $b$, $m$, and $m_0$ are four nonnegative integers.
If $G$ is $4b^2$-tough and $|V(G)|\ge 4b^2$, 
then $G$ has an $m$-tree-connected factor $H$ such that its complement is $m_0$-tree-connected and for each vertex~$v$,
 $$d_H(v)\in \{g(v),f(v)\},$$
provided that either there is a vertex $u$ with $f(u)-g(u)$ odd or $f(v)-g(v)$ is even for all vertices $v$ and $\sum_{v\in V(G)}f(v)$ is even. 
}\end{thm}
\begin{proof}
{For each vertex $v$, define $h(v)=f(v)-m-1$.
Since $3m+2m_0+6k^2<f(v)\le b$, we must have $2m' \le 2h(v) \le b' $, 
where $m'=2m+2m_0+6k^2$ and $b'=2b-2m-2$.
By the assumption, $|V(G)|\ge 4b^2\ge (b')^2$.
Thus by Theorem~\ref{intro:thm:lem:f-f+1}, the graph $G$ has an $m'$-tree-connected $\{2h,2h+1\}$-factor $G'$. 
Since $|V(G)| \ge 4b^2 \ge 4k$ and $m'\ge 2k-2$, 
by Lemma~\ref{lem:factor-bipartite-index}, 
there is a matching $M$ of size $k-1$ such that $bi(G'\cup M)\ge k-1$. Let $G_0=G'\cup M$.
Note that for each vertex $v$, $2h(v)\le d_{G_0}(v) \le 2h(v)+2$. 
Since $m+m_0<f(v)-g(v)$, we must have 
$g(v)+m_0\le h(v)\le \frac{1}{2}d_{G_0}(v)\le h(v) +1= f(v)-m$.
Therefore, by Theorem~\ref{thm:main:g-f}, the graph $G_0$ has an $m$-tree-connected $\{g,f\}$-factor $H$ such that its complement is $m_0$-tree-connected and so does $G$.
}\end{proof}
%
%
%
%
%
%
%
%
%
%
%
%
%
%
%
%
%
%
%

\end{document}